\documentclass[a4paper,leqno,11pt,twoside]{amsart}
\usepackage{amsmath,amsthm,amssymb,amsfonts} 
\usepackage[all]{xy}
\UseTips

\theoremstyle{definition}
			 \newtheorem{rmk}[subsection]{Remark}
			 \newtheorem{rmks}[subsection]{Remarks}

			 \newtheorem{ex}[subsection]{Example}
			 \newtheorem{defn-prop}[subsection]{Definition-Proposition}

\theoremstyle{plain}     \newtheorem{thm}[subsection]{Theorem}

			 \newtheorem{lem}[subsection]{Lemma}
			 \newtheorem{cor}[subsection]{Corollary}
			 \newtheorem{prop}[subsection]{Proposition}

\numberwithin{equation}{subsection}

\newcounter{romain}[subsection]
\newcommand{\romain}
{\stepcounter{romain}
\noindent\makebox[1.5cm][r]{{\normalfont(\roman{romain})}\hspace{0.3cm}}}

\newcommand{\eq}[2]{\begin{equation}\label{#1}#2 \end{equation}}
\newcommand{\eqn}[1]{\begin{equation*}#1\end{equation*}}

\newcommand{\aln}[1]{\begin{align*}#1\end{align*}}
\newcommand{\ml}[2]{\begin{multline}\label{#1}#2 \end{multline}}
\newcommand{\mln}[1]{\begin{multline}#1\end{multline}}
\newcommand{\ga}[2]{\begin{gather}\label{#1}#2 \end{gather}}

\newcommand{\sur}{\twoheadrightarrow}
\newcommand{\inj}{\hookrightarrow}
\newcommand{\lra}{\longrightarrow}
\newcommand{\xra}[1]{\xrightarrow{#1}}
\newcommand{\riso}{\xrightarrow{\,\sim\,}}

\newcommand{\Spec}{{\rm Spec \,}}
\newcommand{\Spf}{{\rm Spf \,}}

\newcommand{\can}{{\rm can}}
\newcommand{\forget}{{\rm forget}}
\newcommand{\restr}{{\rm restr}}

\newcommand{\Ker}{\mathrm{Ker}}

\newcommand{\Id}{\mathrm{Id}}

\newcommand{\sD}{{\mathcal D}}
\newcommand{\sE}{{\mathcal E}}
\newcommand{\sF}{{\mathcal F}}

\newcommand{\sI}{{\mathcal I}}

\newcommand{\sO}{{\mathcal O}}
\newcommand{\sP}{{\mathcal P}}

\newcommand{\sS}{{\mathcal S}}

\newcommand{\sV}{{\mathcal V}}

\newcommand{\sX}{{\mathcal X}}

\newcommand{\NN}{{\mathbb N}}

\newcommand{\ZZ}{\mathbb{Z}}

\newcommand{\fa}{\mathfrak{a}}

\newcommand{\uk}{\underline{k}}
\newcommand{\uq}{\underline{q}}

\newcommand{\del}{\partial}
\newcommand{\udel}{\underline{\del}}
\newcommand{\utau}{\underline{\tau}}

\newcommand{\la}{\langle}
\newcommand{\ra}{\rangle}

\newcommand{\dix}[1]{\mathcal{D}_{#1}^{(\infty)}}
\newcommand{\tx}{\tilde{x}}
\newcommand{\tI}[1]{\widetilde{I^{(#1)}}}
\newcommand{\tsI}[3]{\widetilde{\mathcal{#1}_{#2}^{(#3)}}}
\newcommand{\osI}{\overline{\mathcal{I}}}
\newcommand{\dmx}[2]{\sD^{(#1)}_{#2}}

\begin{document}

\title[Frobenius divided modules]{A note on Frobenius divided modules\\ 
in mixed characteristics}

\author{Pierre Berthelot}
\address{IRMAR, Universit\'e de Rennes 1,
Campus de Beaulieu,
35042 Rennes cedex, France}
\email{pierre.berthelot@univ-rennes1.fr}

\begin{abstract}
If $X$ is a smooth scheme over a perfect field of characteristic $p$, and if $\dix{X}$
is the sheaf of differential operators on $X$ \cite{EGAIV4}, it is well known that
giving an action of $\dix{X}$ on an $\sO_X$-module $\sE$ is equivalent to giving an
infinite sequence of $\sO_X$-modules descending $\sE$ via the iterates of the Frobenius
endomorphism of $X$ \cite{Gi}. We show that this result can be generalized to any
infinitesimal deformation $f : X \to S$ of a smooth morphism in characteristic $p$,
endowed with Frobenius liftings. We also show that it extends to adic formal schemes
such that $p$ belongs to an ideal of definition. In \cite{DS}, dos Santos used this
result to lift $\dix{X}$-modules from characteristic $p$ to characteristic $0$ with
control of the differential Galois group.
\end{abstract}


\date{March 10, 2010.}

\maketitle
\thispagestyle{empty}

\setcounter{tocdepth}{1}\tableofcontents

\section*{Introduction}
Let $X_0$ be a smooth scheme over a perfect field $k$ of characteristic $p>0$, and
$\dix{X_0}$ the sheaf of differential operators on $X_0$ relative to $S_0 = \Spec k$
(in the sense of \cite[16.8]{EGAIV4}). A classical result of Katz \cite[Th.~1.3]{Gi},
based on Cartier's descent \cite[Th.~5.1]{Ka}, asserts that there is an equivalence
between the category of vector bundles on $X_0$ endowed with a left action of
$\dix{X_0}$, and the category of families of vector bundles $\sE_i$ on $X_0$, $i \geq
0$, endowed with $\sO_{X_0}$-linear isomorphisms $\alpha_i : F_{X_0}^*\sE_{i+1} \riso
\sE_i$, where $F_{X_0}$ is the absolute Frobenius endomorphism of $X_0$.
The purpose of this note is to explain how the theory of arithmetic $\sD$-modules
developed in \cite{Be1} and \cite{Be2} allows to generalize this result to
infinitesimal deformations of this setup, which are not necessarily characteristic $p$
deformations. Using limit arguments, we obtain a similar generalization for separated
and complete modules over a formal scheme, including in mixed characteristics. When the
base is a discrete valuation ring of mixed characteristics and $\dim(X_0)=1$, we
recover the correspondence defined earlier by Matzat \cite{Ma}.

We actually start with the more general situation of a smooth morphism $f_0 : X_0 \to
S_0$ between characteristic $p$ schemes. In particular, the perfection hypothesis on
the basis can be removed simply by working with the relative Frobenius morphism
$F_{X_0/S_0}$ instead of the absolute Frobenius endomorphism $F_{X_0}$ (as in
\cite[Th.~5.1]{Ka}). We consider a nilpotent immersion $S_0 \inj S$ and a smooth
morphism $f : X \to S$ lifting $f_0$. We assume that an endomorphism $\sigma : S \to S$
lifting $F_{S_0}$ and an $S$-morphism $F : X \to X^{(1)}$ lifting $F_{X_0/S_0}$ are
given (denoting by $X^{(i)}$ the pull-back of $X$ by $\sigma^i$). Then our main result
is Theorem \ref{Main}, which asserts that, under these assumptions, the category of
$\dix{X}$-modules is equivalent to the category of families of $\sO_{X^{(i)}}$-modules
$\sE_i$ endowed with isomorphisms $F^*\sE_{i+1} \riso \sE_i$. Note that this
equivalence holds without any condition on the modules.

There are two steps in the proof. The first one is to show that, for any such family,
there exists on each $\sE_i$ a unique structure of $\dix{X^{(i)}}$-module such that the
isomorphisms $\alpha_i$ are $\dix{X^{(i)}}$-linear (Theorem \ref{OeqD}). The second one
is to show that a $\dix{X}$-module can be indefinitely descended by liftings of
Frobenius. While the latter is a direct consequence of the Frobenius descent theorem
\cite[2.3.6]{Be2}, the first step is not covered by the results of \cite{Be2}. It
requires the whole structure provided by the infinite sequence $(\sE_i,\alpha_i)$, but
the theory of arithmetic $\sD$-modules provides a more precise information about the
differential structure obtained after a finite number of Frobenius pull-backs. Namely,
the key result is the following (Proposition \ref{GetmPDstrat}): if $\fa \subset \sO_S$
is the ideal defining $S_0$, and if $r$ is an integer such that $\fa^r=0$, then, for
any $m \geq 0$ and any $\sO_{X^{(m+r)}}$-module $\sF$, there exists on $F^{m+r\,*}\sF$
a canonical structure of $\dmx{m}{X}$-module, $\dmx{m}{X}$ being the ring of 
differential operators defined in \cite{Be1}. To prove the existence of this
structure and its main properties, the main tool is the interpretation of $\sD$-module
structures in terms of appropriate notions of stratification, as initiated by
Grothendieck \cite[Appendix]{Gr}.

This note has been written to answer questions raised by J.~P.~dos~Santos in his study
of the variation of the differential Galois group associated to liftings of
$\dix{X}$-modules from charactristic $p$ to characteristic $0$ \cite{DS}. It is a
pleasure to thank him for giving me this opportunity to clarify the relations between
the Frobenius descent theorem proved in \cite{Be2} and the classical interpretations of
$\dix{X}$-modules in terms of infinite Frobenius descent.

\vspace{2mm}
\noindent\textbf{Conventions.} ---\ \ a) We denote by $p$ a fixed prime number.

b) In this note, modules over non commutative rings will always be left modules.

\section{Frobenius divided $\mathcal{D}$-modules and $\mathcal{O}$-modules}
\label{Divided}

We show here that the notions of Frobenius divided $\mathcal{O}$-module and Frobenius 
divided $\mathcal{D}$-module coincide.

\subsection{}\label{Setup}
Let $S$ be a scheme, and $\fa \subset \sO_S$ a quasi-coherent nilpotent ideal such
that $p \in \fa$. We denote by $S_0 \subset S$ the closed subscheme defined by $\fa$,
and we suppose given an endomorphism $\sigma : S \to S$ lifting the absolute Frobenius
endomorphism of $S_0$. For any $S$-scheme $X$ and any $i \in \NN$, we denote by $X_0$
the reduction of $X$ modulo $\fa$, and by $X^{(i)}$ the $S$-scheme deduced from $X$ by
base change by $\sigma^i : S \to S$.

We will consider $S$-schemes $X$ endowed with an $S$-morphism $F : X \to X^{(1)}$
lifting the relative Frobenius morphism of $X_0$ with respect to $S_0$. For any $i \geq
0$, we will simply denote by $F : X^{(i)} \to X^{(i+1)}$ the morphism deduced from $F$
by base change by $\sigma^i$. More generally, for any $i, r \geq 0$, we will denote by
$F^r : X^{(i)} \to X^{(i+r)}$ the composition of the $r$ successive morphisms $F :
X^{(i+j)} \to X^{(i+j+1)}$ for $0 \leq j \leq r-1$. The morphism $F$ is automatically
finite locally free, as a consequence of the flatness criterion by fibers
\cite[Th.~11.3.10]{EGAIV3}.
 
In this situation, an \textit{$F$-divided $\sO_X$-module} will be a family
$(\sE_i,\alpha_i)_{i \geq 0}$ of $\sO_{X^{(i)}}$-modules $\sE_i$, endowed with
$\sO_{X^{(i)}}$-linear isomorphisms $\alpha_i : F^*\sE_{i+1} \riso \sE_i$. They form a
category, for which the morphisms from $(\sE_i,\alpha_i)$ to $(\sE'_i,\alpha'_i)$ are
the families of $\sO_{X^{(i)}}$-linear homomorphisms $\sE_i \to \sE'_i$ which commute
with the $\alpha_i$'s and $\alpha'_i$'s in the obvious sense. Note that our
terminology differs a little from that of \cite{DS}, where the term ``$F$-divided'' is
used only when $F$ is the actual Frobenius in characteristic $p$, and the term
``$\Phi$-divided'' is used instead in a lifted situation of mixed characteristics.
Other terminologies can be found in the literature; I hope the terminology used here
will not cause any confusion.

We will assume that $X$ is a smooth $S$-scheme, so that we can consider the sheaf of
differential operators on $X$ relative to $S$, as defined by Grothendieck in
\cite[16.8]{EGAIV4}. We will denote this sheaf by $\dix{X}$, and we recall that, if $f :
X \to Y$ is an $S$-morphism between two smooth $S$-schemes, the usual inverse image
$f^*\sF$ in the sense of $\sO$-modules of a $\dix{Y}$-module $\sF$ has a canonical
structure of $\dix{X}$-module (see for example \cite[2.1.1]{Be2}, which is valid
for $\dix{X}$-modules). Since $X^{(i)}$ is smooth over $S$ for all $i$, we can
introduce the notion of \textit{$F$-divided $\dix{X}$-module} as being a family
$(\sE_i,\alpha_i)_{i \geq 0}$ of $\dix{X^{(i)}}$-modules $\sE_i$, endowed with
$\dix{X^{(i)}}$-linear isomorphisms $\alpha_i : F^*\sE_{i+1} \riso \sE_i$.

The main result of this section is the following:

\begin{thm}\label{OeqD}
Under the previous hypotheses, the obvious forgetful functor 
\eq{forget}{ \Omega_S : \{ F\textnormal{-divided }\dix{X}\textnormal{-modules} \}  \lra  
\{ F\textnormal{-divided }\sO_X\textnormal{-modules} \} }
is an equivalence of categories.
\end{thm}

More precisely, given an $F$-divided $\sO_X$-module $(\sE_i,\alpha_i)$, there exists on
each $\sE_i$ a unique structure of $\dix{X^{(i)}}$-module such that the isomorphisms
$\alpha_i$ are $\dix{X^{(i)}}$-linear, and each morphism of $F$-divided $\sO_X$-modules
is then a family of $\dix{X^{(i)}}$-linear maps. Endowing each $\sE_i$ with this
$\dix{X^{(i)}}$-module structure provides a quasi-inverse functor to $\Omega_S$.

We will prove this statement in subsection \ref{PfOeqD}, after a few preliminary
results. We first fix notations. If $A$ is a commutative ring, $I \subset A$ an ideal,
and $j \geq 0$ an integer, we denote by $I^{(j)} \subset A$ the ideal generated by the
elements $a^{p^j}$, when $a$ varies in $I$. Assuming that another ideal $\fa \subset A$
has been fixed, we define for each $i \geq 0$
\eq{deftilde}{ \tI{i} := I^{(i)} + \fa I^{(i-1)} + \cdots + \fa^{i-j}I^{(j)} + 
\cdots + \fa^{i}I. }
We will use similar notations for sheaves of ideals.

\begin{lem}\label{Newton}
With the previous notations, assume that $p \in \fa$, and let $x \in \tI{i}$. Then
$x^p$ belongs to $\tI{i+1}$.
\end{lem}

\begin{proof}
We can write $x$ as $x = y_i + \cdots + y_0$, with $y_j \in \fa^{i-j}I^{(j)}$.
Therefore,
\eqn{ x^p \in (y_i^p,\ldots,y_0^p) + p(y_i,\ldots,y_0). }
On the one hand, we have
\eqn{ py_j \in p\fa^{i-j}I^{(j)} \subset \fa^{i-j+1}I^{(j)} \subset \tI{i+1}. }
On the other hand, for each $j$ such that $0 \leq j \leq i$, we can write $y_j$ as a 
sum $y_j = \sum_{k=1}^{r_j} a_{j,k}z_{j,k}^{p^j}$, with $a_{j,k} \in \fa^{i-j}$ and
$z_{j,k} \in \sI$. We deduce
\eqn{ y_j^p = \sum_{k=1}^{r_j} a_{j,k}^p z_{j,k}^{p^{j+1}} + 
\sum_{\substack{n_1+\cdots+n_{r_j}=p\\ \forall k,\;0\leq n_k\neq p}} 
\binom{p}{n_1,\ldots,n_{r_j}} \prod_{k=1}^{r_j}a_{j,k}^{n_k}z_{j,k}^{n_k p^j}. }
Each term of the first sum belongs to $\fa^{i-j}I^{(p^{j+1})} \subset \tI{i+1}$. In 
each term of the second sum, the multinomial coefficient is divisible by $p \in \fa$, 
and the product belongs to $\fa^{i-j}I^{(j)}$. So the second sum too belongs to 
$\tI{i+1}$.
\end{proof}

\begin{lem}\label{Diag}
Under the hypotheses of \ref{Setup}, let $r \geq 0$ be an integer, and let $\sI_r
\subset \sO_{X^{(r)}\times X^{(r)}}$ be the ideal defining the diagonal immersion
$X^{(r)} \inj X^{(r)}\times_S X^{(r)}$. Then, for $i \leq r$, we have
\eq{keyincl}{ (F^i\times F^i)^*(\sI_r) \subset \tsI{I}{r-i}{i}. }
\end{lem}

\begin{proof}
The ideal $\sI_r$ is generated by sections of the form $\xi' = 1\otimes x' - x'\otimes
1$, where $x'$ is a section of $\sO_{X^{(r)}}$, and we may assume that $x' = 1\otimes 
x$, where $x$ is a section of $\sO_{X^{(r-i)}}$. Let $\xi = 1\otimes x - x\otimes 1 
\in \sO_{X^{(r-i)}\times X^{(r-i)}}$. 

For $i = 1$, the morphism $F : X^{(r-1)} \to X^{(r)}$ is a lifting over $S$ of the 
relative Frobenius morphism of $X^{(r-1)}$. So we can write
\eqn{ F^*(x') = x^p + \sum_k a_k y_k, }
with $a_k \in \fa$, $y_k \in \sO_{X^{(r-1)}}.$ It follows that
\aln{ (F\times F)^*(\xi') & = 1 \otimes x^p - x^p \otimes 1 + \sum_k a_k(1\otimes y_k
- y_k \otimes 1)\\
& = (\xi + x\otimes 1)^p - x^p \otimes 1 + \sum_k a_k(1\otimes y_k - y_k \otimes 1)\\ 
& = \xi^p + \sum_{k=1}^{p-1}\binom{p}{k} (x^{p-k}\otimes 1) \xi^k + 
\sum_k a_k(1\otimes y_k - y_k \otimes 1). }
Since $p \in \fa$, both sums belong to $\fa\sI_{r-1}$, and we get that $(F\times 
F)^*(\xi')$ belongs to $\sI_{r-1}^{(1)} + \fa\sI_{r-1} = \tsI{I}{r-1}{1}$ as wanted.

We can then argue by induction on $i$. Assuming the lemma for $i-1$, it suffices to 
prove that 
\eq{raisetsI}{(F\times F)^*(\tsI{I}{r-i+1}{i-1}) \subset \tsI{I}{r-i}{i}. }
Since $\fa^{i-j-1}\tsI{I}{r-i}{j+1} \subset \tsI{I}{r-i}{i}$, it suffices to show that,
for any $\eta \in \sI_{r-i+1}$ and any $j \geq 0$, $(F\times F)^*(\eta^{p^j}) \in
\tsI{I}{r-i}{j+1}$. But $(F\times F)^*(\eta^{p^j}) = ((F\times
F)^*(\eta))^{p^j}$, and we have seen that $(F\times F)^*(\eta) \in
\tsI{I}{r-i}{1}$. Applying repeatedly Lemma \ref{Newton}, the claim follows.
\end{proof}

\subsection{}\label{Dmodar}
We now recall briefly some notions about arithmetic $\sD$-modules; readers looking for 
an introduction with more details can refer to \cite{CEB}.

On an open subset where $X$ has a set of local coordinates $t_1,\ldots,t_d$ relative 
to $S$, the ring $\dix{X}$ is a free $\sO_X$-module which admits as basis a set of 
operators $\udel^{[\uk]}$, for $\uk = (k_1,\ldots,k_d) \in \NN^d$, satisfying the 
following properties:

\romain If $\uk = (0,\ldots,0)$, then $\udel^{[\uk]} = 1$.

\romain If $k_i=1$, and $k_j=0$ for $j\neq i$, then $\udel^{[\uk]}$ is the derivation 
$\del/\del t_i$ from the dual basis to the basis of $1$-forms $(dt_j)$. 

\romain For all $\uk,\uk' \in \NN^d$, we have $\udel^{[\uk]}\udel^{[\uk']} = 
\binom{\uk+\uk'}{\uk}\udel^{[\uk+\uk']}$.

The last property shows that the operators $\udel^{[\uk]}$ behave ``as
$\frac{1}{\uk!}\prod_i(\del/\del t_i)^{k_i}$ in characteteristic $0$''. It also
shows that, outside characteristic $0$, prescribing the action of the derivations
$\del/\del t_i$ does not suffice to define an action of $\dix{X}$. It is well known
that such an action is determined by the action of the operators $(\del/\del
t_i)^{[p^j]}$ for all $i$ and all $j \in \NN$. In particular, prescribing an action of
$\dix{X}$ on an $\sO_X$-module is a process of infinite nature.

When the integers prime to $p$ are invertible on the base scheme (as in our situation),
one way to do it is to use the ``rings of differential operators of finite level''
$\dmx{m}{X}$, for $m \in \NN$ \cite{Be1}: these form a direct system of rings such that
\eq{limit}{ \varinjlim_m \dmx{m}{X} \riso \dix{X}. }
In a local situation as above, $\dmx{m}{X}$ is a free $\sO_X$-module. It has a basis
of operators $\udel^{\la\uk\ra_{(m)}}$ such that $\udel^{\la\uk\ra_{(m)}}$
maps to $\uq! \udel^{[\uk]}$, where $\uq = (q_1,\ldots,q_d)$ is defined by
\eq{euclid}{ \forall i, \quad k_i = p^m q_i + r_i, \quad 0 \leq r_i < p^m.}
In particular, the homomorphism $\dmx{m}{X} \to \dix{X}$ is not injective when $p$ is
nilpotent on $S$, but it induces an isomorphism of $\sO_X$-modules between the
subsheaves of differential operators of order $< p^{m+1}$ (note in particular that, for
$j \leq m$, $(\del/\del t_i)^{\la p^j\ra_{(m)}}$ maps to $(\del/\del t_i)^{[p^j]}$ in
$\dix{X}$). An important difference between the sheaves $\dmx{m}{X}$ and $\dix{X}$ is
that an action of $\dmx{m}{X}$ on an $\sO_X$-module is known when the action of the
operators $(\del/\del t_i)^{\la p^j\ra_{(m)}}$ is known for $j \leq m$ (this is a
consequence of the decomposition of the operators $\udel^{\la\uk\ra_{(m)}}$
\cite[(2.2.5.1)]{Be1}). Thus prescribing an action of $\dmx{m}{X}$ is a process of
finite nature, and this will be illustrated by Proposition \ref{GetmPDstrat} below.

To define the action of differential operators of order $> 1$, we will have to use the
notion of stratification in the case of $\dix{X}$-modules (see \cite[2.10]{BO}) and its
divided power variants in the case of $\dmx{m}{X}$-modules (see \cite[4.3]{BO} and its
level $m$ generalization \cite[2.3]{Be1}). If $\sI$ is the ideal defining the diagonal
immersion $X \inj X\times_S X$, we will denote by $\sP^n_{X}$ the sheaf $\sO_{X\times
X}/\sI^{n+1}$ (sheaf of principal parts of order $n$ on $X$), by $\sP_{X,(m)}$ the
$m$-PD-envelope of $\sI$ \cite[Prop.~1.4.1]{Be1}, by $\osI$ the canonical
$m$-PD-ideal defined by $\sI$ in $\sP_{X,(m)}$, and by $\sP^n_{X,(m)}$ the quotient
$\sP_{X,(m)}/\osI^{\{n+1\}_{(m)}}$, where $\osI^{\{n+1\}_{(m)}}$ is the $m$-th step of
the $m$-PD-adic filtration (as defined in \cite[App., A.3]{Be2}). We recall that a
stratification (resp.~an $m$-PD-stratification) on an $\sO_X$-module $\sE$ is a family
of linear isomorphisms
\ga{strat}{ \varepsilon_n : \sP^n_{X} \otimes_{\sO_X} \sE \riso \sE \otimes_{\sO_X}
\sP^n_X \\
\label{mstrat}\text{(resp.} \ \varepsilon_n : 
\sP^n_{X,(m)} \otimes_{\sO_X} \sE \riso \sE \otimes_{\sO_X} \sP^n_{X,(m)}),
}
compatible when $n$ varies, such that $\varepsilon_0 = \Id$, and satisfying a cocycle
relation on the triple product $X\times_S X\times_S X$ (the $\sO_X$-algebra structures
used on $\sP^n_X$ for the source and target of $\varepsilon_n$ are defined respectively
by the second and first projections $X\times_S X\to X$). Then the datum of a structure
of left $\dix{X}$-module (resp.~$\dmx{m}{X}$-module) on an $\sO_X$-module $\sE$,
extending its $\sO_X$-module structure, is equivalent to the datum of a stratification
\cite[Prop.~2.11]{BO} (resp.~an $m$-PD-stratification \cite[Prop.~2.3.2]{Be1}). The
relation between these two types of data is made explicit by the ``Taylor formula'',
which describes the isomorphisms $\varepsilon_n$ in local coordinates,
\eq{Taylor}{ \forall x \in \sE,\quad \varepsilon_n(1\otimes x) = \sum_{|\uk|\leq 
n}(\udel^{[\uk]}x\otimes 1)\utau^{\uk}, }
and by its analogue for $\dmx{m}{X}$-modules \cite[2.3.2]{Be1} (here, we have set 
$\tau_i = 1 \otimes t_i - t_i\otimes 1$, and $|\uk| = k_1+\cdots+k_d$).

\begin{prop}\label{OtoD}
Under the hypotheses of \ref{Setup}, let $r \geq 1$ be an integer such that $\fa^r =
0$, and let $m \in \NN$ be another integer. 

\romain There exists a (unique) ring homomorphism
$\varphi_{m,r} : \sO_{X^{(m+r)}} \to \sP_{X,(m)}$ such that the diagram 
\eq{factorize}{ \xymatrix@C=5ex{
\sO_{X^{(m+r)}} \ar[d]_{F^{m+r\,*}} \ar@<.8ex>[r] \ar@<-.6ex>[r] &
\sO_{X^{(m+r)}\times X^{(m+r)}} \ar[d]_{(F^{m+r}\times F^{m+r})^*} \ar[r] &
\sP_{X^{(m+r)},(m)} \ar[d]_{(F^{m+r}\times F^{m+r})^*} \ar@{->>}[r] & 
\sO_{X^{(m+r)}} \ar[ld]_{\varphi_{m,r}} \ar[d]^{F^{m+r\,*}}  \\
\sO_X \ar@<.8ex>[r] \ar@<-.6ex>[r] & \sO_{X\times X} \ar[r] & 
\sP_{X,(m)} \ar@{->>}[r]  & \sO_X 
} }
commutes.

\romain If $m' \geq m$, the square
\eq{incm}{ \xymatrix@C=8ex{
\sO_{X^{(m'+r)}} \ar[d]_-{F^{m'-m\,*}} \ar[r]^-{\varphi_{m',r}} & 
\sP_{X,(m')} \ar[d]^-{\can} \\
\sO_{X^{(m+r)}} \ar[r]+<-4.1ex,0ex>^-{\varphi_{m,r}} & 
\mspace{2mu}\sP_{X,(m)},\mspace{-2mu}
} }
where the right vertical arrow is the canonical homomorphism \cite[1.4.7]{Be1}, is
commutative.

\romain If $s \geq r$, then the homomorphisms $\varphi_{m,r} \circ F^{s-r\,*} :
\sO_{X^{(m+s)}} \to \sO_{X^{(m+r)}} \to \sP_{X,(m)}$ and $(F^{s-r}\times F^{s-r})^*
\circ \varphi_{m,r} : \sO_{X^{(m+s)}} \to \sP_{X^{(s-r)},(m)} \to \sP_{X,(m)}$ are both
equal to $\varphi_{m,s}$. 
\end{prop}

\begin{proof}
The kernel of the homomorphism $\sP_{X^{(m+r)},(m)} \sur \sO_{X^{(m+r)}}$ is the 
$m$-PD-ideal generated by $\sI_{m+r}$. Since $(F^{m+r}\times F^{m+r})^*$ is an 
$m$-PD-morphism, the factorization $\varphi_{m,r}$ exists if and only if the image of 
$\sI_{m+r}$ in $\sP_{X,(m)}$ is $0$. Using Lemma \ref{Diag}, it suffices to prove that 
the image of $\tsI{\sI}{}{m+r}$ in $\sP_{X,(m)}$ is $0$. 

An element $x \in \tsI{\sI}{}{m+r}$ can be written $x=\sum_{j=0}^{m+r}\sum_k 
a_{j,k}z_{j,k}^{p^j}$, with $a_{j,k} \in \fa^{m+r-j}\sO_X$ and $z_{j,k} \in \sI$. 
Since $\fa^r=0$, $a_{j,k} = 0$ for $j \leq m$. On the other hand, $\sI$ is mapped by 
construction to the $m$-PD-ideal $\osI \subset \sP_{X,(m)}$. Such an ideal is 
equipped with partial divided power operations $z \mapsto z^{\{k\}_{(m)}}$, $k \in 
\NN$, such that $z^k = q!z^{\{k\}_{(m)}}$ where $q$ is the quotient of $k$ by $p^m$ as 
in \eqref{euclid} \cite[1.3.5]{Be1}. For $j>m$, we obtain $z_{j,k}^{p^j} = 
p^{j-m}!z_{j,k}^{\{k\}_{(m)}}$. Since $p \in \fa$, $a_{j,k}z_{j,k}^{p^j}$ maps to 
$\fa^{m+r-j+j-m}\sP_{X,(m)} = 0$. This shows the first assertion. 

To prove the second one, it suffices to show that the square commutes after composing
with the surjection $\sO_{X^{(m'+r)}\times X^{(m'+r)}} \sur \sO_{X^{(m'+r)}}$. Thanks
to \eqref{factorize}, the composition of this surjection with $\can \circ \varphi_{m',r}$
is equal to
\eqn{ \sO_{X^{(m'+r)}\times X^{(m'+r)}} \xra{(F^{m'+r}\times F^{m'+r})^*} \sO_{X\times X} 
\lra \sP_{X,(m)}, }
while its composition with $\varphi_{m,r} \circ F^{m'-m\,*}$ is equal to 
\mln{ \sO_{X^{(m'+r)}\times X^{(m'+r)}} \xra{(F^{m'-m}\times F^{m'-m})^*} 
\sO_{X^{(m+r)}\times X^{(m+r)}} \\
\xra{(F^{m+r}\times F^{m+r})^*} \sO_{X\times X} 
\lra \sP_{X,(m)}. }
The assertion follows, and the third one is proved similarly. 
\end{proof}

\begin{prop}\label{GetmPDstrat}
Under the hypotheses of \ref{Setup}, let $r \geq 1$ be an integer such that $\fa^r =
0$, and let $m \in \NN$ be another integer. 

\romain For any $\sO_{X^{(m+r)}}$-module $\sF$, there exists on $F^{m+r\,*}\sF$ a 
canonical $\dmx{m}{X}$-module structure extending its $\sO_X$-module structure. This 
structure is functorial with respect to $\sF$, and the functor $\Phi_{m,r}$ defined in 
this way from the category of $\sO_{X^{(m+r)}}$-modules to the category of 
$\dmx{m}{X}$-modules fits in a commutative diagram of functors
\eq{caractPhi}{ \xymatrix@C=10ex{
\{\sO_{X^{(m+r)}}\textnormal{-modules}\} \ar[r]^-{F^{m+r\,*}} \ar[rd]^{\Phi_{m,r}} &
\{ \sO_X\textnormal{-modules} \} \\
\{\dmx{m}{X^{(m+r)}}\textnormal{-modules}\} \ar[r]+<-8.7ex,0ex>^-{F^{m+r\,*}} 
\ar[u]^-{\forget} & 
\mspace{4mu}\{\dmx{m}{X}\textnormal{-modules}\}.\mspace{-4mu} \ar[u]_-{\forget}
} }

\romain If $m' \geq m$, the diagram of functors 
\eq{incm2}{ \xymatrix@C=10ex{
\{\sO_{X^{(m'+r)}}\textnormal{-modules}\} \ar[r]^-{\Phi_{m',r}} \ar[d]_-{F^{m'-m\,*}} & 
\{\dmx{m'}{X}\textnormal{-modules}\} \ar[d]^-{\restr} \\
\{\sO_{X^{(m+r)}}\textnormal{-modules}\} \ar[r]^-{\Phi_{m,r}} & 
\{\dmx{m}{X}\textnormal{-modules}\}
} }
commutes up to canonical isomorphism.

\romain If $s \geq r$, then $\Phi_{m,s} \simeq \Phi_{m,r} \circ F^{s-r\,*} \simeq
F^{s-r\,*} \circ \Phi_{m,r}$, where the last functor $F^{s-r\,*}$ is the inverse image
functor for $\dmx{m}{X^{(s-r)}}$-modules.

\end{prop}

\begin{proof}
To define a $\dmx{m}{X}$-module structure on $F^{m+r\,*}\sF$, we endow it with an 
$m$-PD-stratification as follows.
 
For each $i \geq 0$, let $P_{X^{(i)},(m)} = \Spec\sP_{X^{(i)},(m)}$, and let $p_0, p_1
: P_{X^{(i)},(m)} \to X^{(i)}$ be the morphisms induced by the two projections
$X^{(i)}\times X^{(i)} \to X^{(i)}$. Using the morphism $\phi_{m,r} : P_{X,(m)} \to
X^{(m+r)}$ defined by the homomorphism $\varphi_{m,r}$ provided by Proposition
\ref{OtoD}, and denoting $\Delta : X^{(i)} \inj P_{X^{(i)},(m)}$ the factorizations of
the diagonal immersions (defined by $\sP_{X^{(i)},(m)} \sur \sO_{X^{(i)}}$), we obtain
isomorphisms
\eq{canstrat}{ p_1^*(F^{m+r\,*}\sF) \riso (F^{m+r}\times F^{m+r})^*p_1^*\sF \riso 
\phi_{m,r}^*\Delta^*p_1^*\sF \riso \phi_{m,r}^*\sF, }
and similarly for $p_0^*(F^{m+r\,*}\sF)$. Composing and reducing mod
$\osI^{\{n+1\}_{(m)}}$ for all $n$, we obtain a compatible family of isomorphisms
$\varepsilon_n : p_1^*(F^{m+r\,*}\sF) \riso p_0^*(F^{m+r\,*}\sF)$ as in \eqref{mstrat}.
To prove that they define an $m$-PD-stratification, one must check the cocycle 
condition. This is formal, once one has checked that Proposition \ref{OtoD} (i) can be 
generalized to the product $X \times_S X\times_S X$. Since the argument is the same 
than for the proof of \ref{OtoD} (i), we omit the details. 

This $m$-PD-stratification is clearly functorial in $\sF$. Endowing $F^{m+r\,*}\sF$ 
with the corresponding $\dmx{m}{X}$-module structure defines the functor $\Phi_{m,r}$, 
and the upper triangle in \eqref{caractPhi} commutes by construction. To prove that 
the lower one commutes, one must check that, if $\sF$ has a 
$\dmx{m}{X^{(m+r)}}$-module structure corresponding to an $m$-PD-stratification 
$(\delta_n)_{n\geq 0}$, the $m$-PD-stratification on $F^{m+r\,*}\sF$ deduced from 
$(\delta_n)$ by scalar extension via the homomorphisms $\sP^n_{X^{(m+r)},(m)} \to 
\sP^n_{X,(m)}$ defined by $F^{m+r}\times F^{m+r}$ coincides with $(\varepsilon_n)$. 
This is a consequence of the commutativity of diagram \eqref{factorize}. 

If $m' \geq m$, an $m'$-PD-stratification defines an $m$-PD-stratification by scalar
extension via the canonical homomorphisms $\sP^n_{(m')} \to \sP^n_{(m)}$. From the
point of view of $\sD$-module structures, this corresponds to scalar restriction from
$\dmx{m'}{}$ to $\dmx{m}{}$ (this can be easily checked using the corresponding Taylor
formulas for $\dmx{m}{}$ and $\dmx{m'}{}$-modules). Therefore, the commutativity of 
\eqref{incm} implies the commutativity of \eqref{incm2}. 

Finally, the third statement follows immediately from Proposition \ref{OtoD} (iii), 
using the definition of the inverse image functor for $\dmx{m}{}$-modules in terms 
of $m$-PD-stratifications \cite[2.1.1]{Be2}. 
\end{proof}

\begin{cor}\label{Getconn}
Under the hypotheses of \ref{Setup}, let $r \geq 1$ be an integer such that $\fa^r =
0$. If $\sF$ is any $\sO_{X^{(r)}}$-module, $F^{r\,*}\sF$ is endowed with a canonical
integrable connexion, functorial in $\sF$. If $\sF$ itself is endowed with an
integrable connexion $\nabla$, the inverse image of $\nabla$ by $F^r$ is the canonical
connexion of $F^{r\,*}\sF$.
\end{cor}

\begin{proof}Since the datum of an integrable conexion is equivalent to the datum of a
$\dmx{0}{}$-module structure \cite[Th.~4.8]{BO}, this is the particular case of
Proposition \ref{GetmPDstrat} (i) obtained for $m = 0$.
\end{proof}

\begin{ex}
Assume that $\fa = 0$, so that $S$ is a characteristic $p$ scheme, and $F$ is the
relative Frobenius morphism $F_{X/S}$. Then we may take $r=1$, and the corollary gives
the classical connexion on the pull-back by $F_{X/S}$ of any $\sO_{X^{(1)}}$-module
\cite[Th.~5.1]{Ka} (see also \cite[2.6]{Be2}). Note that, in this case, Cartier's
theorem provides a characterization of the essential image of $F_{X/S}^*$ as being the
subcategory of the category of $\sO_X$-modules with integrable connexion such that the
connexion has $p$-curvature $0$. It would be interesting to have a similar
characterization of the essential image of $F^{r\,*}$ in the more general situation of
Corollary \ref{Getconn}, at least in the case where $S$ is flat over $\ZZ/p^r\ZZ$, and
$\fa = p\sO_S$.

Assuming again that $\fa = 0$, but for arbitrary $m$, Proposition \ref{GetmPDstrat} is
then a consequence of the combination of the previous remark with
\cite[Prop.~2.2.3]{Be2}, which grants that, for any $m, s \geq 0$, the inverse image of
a $\dmx{m}{X^{(s)}}$-module by $F^s_{X/S}$ has a canonical structure of
$\dmx{m+s}{X}$-module.
\end{ex}

\subsection{}\label{PfOeqD}\textit{Proof of Theorem \ref{OeqD}.} 
We fix an integer $r$ such that $\fa^r=0$. 

\romain \textit{Unicity}. Let $(\sE_i,\alpha_i)$ be an $F$-divided $\dix{X}$-module. To 
prove that the $\dix{X^{(i)}}$-module structure of each $\sE_i$ is determined by the 
family $(\sE_i,\alpha_i)$ viewed as an $F$-divided $\sO_X$-module, we can use the 
isomorphism \eqref{limit} to reduce to proving the, for each $m$ and each $i$, the 
underlying $\dmx{m}{X^{(i)}}$-module structure of $\sE_i$ is determined by the 
$F$-divided $\sO_X$-module $(\sE_i,\alpha_i)$. 

Composing the isomorphisms $\alpha_i$, one gets a $\dmx{m}{X^{(i)}}$-linear 
isomorphism $\psi_{m,r,i} : F^{m+r\,*}\sE_{i+m+r} \riso \sE_i$. By Proposition 
\ref{GetmPDstrat} (i), the $\dmx{m}{X^{(i)}}$-module structure of $F^{m+r\,*}\sE_{i+m+r}$ 
does not depend on the $\dmx{m}{X^{(i+m+r)}}$-module structure of $\sE_{i+m+r}$, and 
is equal to the canonical structure we have defined on the inverse image by $F^{m+r}$ 
of an $\sO$-module. This proves the unicity.

\romain \textit{Existence}. We can again use the isomorphism \eqref{limit} to reduce to 
defining for all $i$ and $m$ a $\dmx{m}{X^{(i)}}$-module structure on $\sE_i$ so that 
$\alpha_i$ is $\dmx{m}{X^{(i)}}$-linear, and so that, for $m' \geq m$, the 
$\dmx{m}{X^{(i)}}$-module structure is induced by the $\dmx{m'}{X^{(i)}}$-module 
structure.

To statisfy the first condition, it suffices to endow $\sE_i$ with the 
$\dmx{m}{X^{(i)}}$-module structure deduced by transport via $\psi_{m,r,i}$ from the 
canonical structure of $F^{m+r\,*}\sE_{i+m+r}$: the fact that, with this definition, 
$\alpha_i$ is $\dmx{m}{X^{(i)}}$-linear follows then from Proposition \ref{GetmPDstrat} 
(iii). As to the second condition, it is a consequence of Proposition \ref{GetmPDstrat} 
(ii).\hfill$\Box$

\begin{rmks}\label{Direct}\hspace{-5mm}\romain Assuming that $S$ is affine, the same
result holds when $X$ is the spectrum of a localization of some smooth $\Gamma(S, 
\sO_S)$-algebra.

\romain One can also construct directly the $\dix{X^{(i)}}$-module structure of $\sE_i$ by 
observing that Lemma \ref{Diag} implies that 
\eqn{ (F^{m+r}\times F^{m+r})^* : \sO_{X^{(i+m+r)}\times 
X^{(i+m+r)}}/\sI_{i+m+r}^{p^{m+1}} \to \sO_{X^{(i)}\times 
X^{(i)}}/\sI_{i}^{p^{m+1}} }
can be factorized through $\sO_{X^{(i+m+r)}}$ as in the construction of $\varphi_{m,r}$. 
Arguing as in the proof of Proposition \ref{GetmPDstrat}, one gets the isomorphisms 
$\varepsilon_n$ of the stratification of $\sE_i$ for $n < p^{m+1}$. Letting $m$ go 
tend to infinity, one gets the whole stratification. This allows to prove Theorem 
\ref{OeqD} without using $\dmx{m}{X}$-module structures as intermediates. However, 
this method does not provide informations on the differential structure of pull-backs 
by a finite iteration of $F$ similar to those provided by Proposition \ref{GetmPDstrat}. 
\end{rmks}

\subsection{}\label{Horiz}
We end this section with a result showing how to compute explicitly the
$\dmx{m}{X}$-module structure defined by Proposition \ref{GetmPDstrat} on
$F^{m+r\,*}\sF$, for an arbitrary $\sO_{X^{(m+r)}}$-module $\sF$.

Let $\sE$ be a $\dmx{m}{X}$-module, and let $x$ be a section of $\sE$. We will say 
that $x$ is \textit{horizontal} if, for all $n \geq 0$, we have
\eq{horiz1}{ \varepsilon_n(1 \otimes x) = x \otimes 1, }
where $(\varepsilon_n)_{n\geq 0}$ is the $m$-PD-stratification corresponding to the 
$\dmx{m}{X}$-module structure on $\sE$. On an open subset endowed with local 
coordinates, this condition can be expressed using the basis of operators 
$\udel^{\la\uk\ra_{(m)}}$: it follows from the Taylor formula that $x$ is horizontal 
if and only if 
\eq{horiz2}{ \forall \uk \neq \underline{0}, \quad \udel^{\la\uk\ra_{(m)}}\cdot x = 0. }
Thanks to the decomposition formula \cite[2.2.5.1]{Be1}, this is equivalent to the 
condition
\eq{horiz3}{ \forall i,\ \ \forall j \leq m, \quad \del_i^{\la p^j\ra_{(m)}}\cdot x = 0. }

\begin{prop}\label{Fhoriz}
Under the assumptions of Proposition \ref{GetmPDstrat}, let $\sF$ be an
$\sO_{X^{(m+r)}}$-module. Then the extension of scalars $\sF \to F^{m+r\,*}\sF$ maps
the sections of $\sF$ to horizontal sections of $F^{m+r\,*}\sF = \Phi_{m,r}(\sF)$.
\end{prop}

\begin{proof}
Let $x'$ be a section of $\sF$, and $x = F^{m+r\,*}(x') \in F^{m+r\,*}\sF$. Then, if
$(\varepsilon_n)$ is the $m$-PD-stratification of $F^{m+r\,*}\sF$, $\varepsilon_n$ is
the reduction mod $\osI^{\{n+1\}}$ of the composition of \eqref{canstrat} with the
inverse of the analog of \eqref{canstrat} starting from $p_0^*(F^{m+r\,*}\sF)$. Via 
\eqref{canstrat}, $1\otimes x$ maps to $\phi_{m,r}^*(x')$, and the same holds for the 
image of $x\otimes 1$ via the inverse of the analog of \eqref{canstrat} for 
$p_0^*(F^{m+r\,*}\sF)$. Therefore, $\varepsilon_n(1\otimes x) = x\otimes 1$, which 
proves the proposition.
\end{proof}

\begin{rmks}\label{Explicit}\hspace{-5mm}\romain Locally, a section of
$F^{m+r\,*}\sF$ can be written as $x = \sum_i a_i\otimes x'_i$, where $a_i \in \sO_X$
and $x'_i \in \sF$. Together with the Leibniz formula \cite[Prop.~2.2.4, (iv)]{Be1},
the previous proposition implies that, for any operator $P \in \dmx{m}{X}$, the action
of $P$ on $x$ is given by
\eq{explicit}{ P\cdot x = \sum_i P(a_i)x_i, } 
with $x_i = F^{m+r\,*}(x'_i)$. 

\romain We can apply the previous proposition to $\sF = \sO_{X^{(m+r)}}$, and this shows 
that the homomorphism $\sO_{X^{(m+r)}} \to \sO_X \to \dmx{m}{X}$ maps 
$\sO_{X^{(m+r)}}$ to the center of $\dmx{m}{X}$. So, for any operator $P \in 
\dmx{m}{X}$, we can let $P$ operate on $F^{m+r\,*}\sF$ by $P \otimes \Id_{\sF}$. 
Formula \eqref{explicit} shows that the $\dmx{m}{X}$-module structure obtained in this 
way on $F^{m+r\,*}\sF$ is the one defined by Proposition \ref{GetmPDstrat}.

\item When $\fa = 0$ and $m=0$, the homomorphism $\sF \to \sE := F^*\sF$ identifies
$\sF$ with the subsheaf $\sE^{\nabla}$ of horizontal sections of $\sE$
\cite[Th.~5.1]{Ka}. This is no longer true in general. For example, let $k$ be a 
perfect field of characteristic $p$, $S = \Spec W_r(k)$, endowed with its natural
Frobenius action, and $\fa = p\sO_S$. If $X$ is a smooth $S$-scheme, and $m=0$,
$\sO_X^{\nabla} = \Ker(d : \sO_X \to \Omega^1_X)$ can be identified with the sheaf 
$W_r\sO_{X_0}$ of Witt vectors of length $r$ on $X_0$ \cite[III, 1.5]{IR}, thanks to 
the map
\eqn{(x_0,\ldots, x_{r-1}) \mapsto \tx_0^{p ^r} + \cdots + p^{r-1}\tx_{r-1}^p, }
where $\tx_i$ is any section $\sO_X$ lifting $x_i$. In particular, if $\dim(X/S) \geq
1$ and $r \geq 2$, $\sO_X^{\nabla}$ is not flat over $S$, hence cannot be identified
with $\sO_{X^{(r)}}$.
\end{rmks}

\section{Frobenius descent}

We now apply the theory of Frobenius descent developed in \cite{Be2} to the case of 
$\dix{X}$-modules.

\begin{thm}\label{Descent}
Under the hypothese of \ref{Setup}, the functor $F^*$ defines an equivalence of 
categories between the category of $\dix{X^{(1)}}$-modules and the category of 
$\dix{X}$-modules.
\end{thm}

\begin{proof}Using \eqref{limit}, we can again view a $\dix{X}$-module as an
$\sO_X$-module endowed for all $m \in \NN$ with a structure of $\dmx{m}{X}$-module
extending its $\sO_X$-module structure, in a compatible way when $m$ varies.

Let $r$ be such that $\fa^r = 0$, and let $m_0$ be an integer such that $p^{m_0}\geq
r$. We endow the ideal $p\sO_S \subset \fa$ with its canonical PD-structure. Because
any section $a \in \fa$ is such that $a^{p^{m_0}} = 0$, and $p\fa \subset p\sO_X$, the
PD-ideal $p\sO_X$ defines for all $m \geq m_0$ an $m$-PD-structure on $\fa$ (called the
\textit{trivial} $m$-PD-structure, see \cite[1.3.1, Ex.~(ii)]{Be1}). So we can apply
the Frobenius descent theorem \cite[Th.~2.3.6]{Be2}, and we obtain that, for each $m
\geq m_0$, the functor $F^*$ defines an equivalence of categories between the category
of $\dmx{m}{X^{(1)}}$-modules and the category of $\dmx{m+1}{X}$-modules. Moreover, for
$m' \geq m$, the remark of \cite[Prop.~2.2.3]{Be2} implies that these equivalences
commute with restrictions of scalars from $\dmx{m'}{X^{(1)}}$ to $\dmx{m}{X^{(1)}}$, 
and from $\dmx{m'+1}{X}$ to $\dmx{m+1}{X}$. The theorem follows formally. 
\end{proof}

\begin{rmk}
Instead of deducing Theorem \ref{Descent} from the Frobenius descent theorem for
$\dmx{m}{X}$-modules, one could also prove it directly by repeating for stratifications
the arguments developed in the proof of \cite[Th.~2.3.6]{Be2} for
$m$-PD-stratifications. Let us recall that the proof builds upon the fact that $F$ is a
finite locally free morphism, as are the morphisms induced between appropriate
infinitesimal neighbourhoods of $X$ and $X^{(1)}$ in their products over $S$. Then
one shows that the stratification induces a descent datum from $X$ to $X^{(1)}$, which
allows to descend the $\sO_X$-module underlying a $\dix{X}$-module as an
$\sO_{X^{(1)}}$-module. The same type of argument allows to descend the isomorphisms
$\varepsilon_n$ defining the stratification corresponding to the $\dix{X}$-module
structure. Finally, using fppf descent again, one shows that the cocycle condition for
the $\varepsilon_n$'s implies the cocycle condition for the descended isomorphisms.
\end{rmk}

\begin{cor}\label{Divide}
Under the hypotheses of $\ref{Setup}$, the functor 
\eq{divide}{ \{ F\textnormal{-divided }\dix{X}\textnormal{-modules} \} \lra 
\{ \dix{X}\textnormal{-modules} \}, \quad\quad (\sE_i,\alpha_i) \mapsto \sE_0, }
is an equivalence of categories.
\end{cor}

\begin{proof}
It suffices to apply repeatedly Theorem \ref{Descent}.
\end{proof}

\begin{thm}\label{Main}
Under the hypotheses of \ref{Setup}, there exists an equivalence of categories 
\eq{delta}{ \Delta_S : \{ F\textnormal{-divided }\sO_X\textnormal{-modules} \}
\xra{\ \approx\ } \{ \dix{X}\textnormal{-modules} \}   
 }
which satisfies the following properties:

\romain If $(\sE_i,\alpha_i)$ is an $F$-divided $\sO_X$-module, the $\sO_X$-module 
underlying $\Delta_S(\sE_i,\alpha_i)$ is $\sE_0$. 

\romain $\Delta_S$ is exact, $\sO_S$-linear, and commutes with tensor products in both 
categories.

\romain If $(S',\fa',\sigma')$ statisfies the conditions of \ref{Setup}, if $u :
(S',\fa',\sigma')\to (S,\fa,\sigma)$ is a morphism commuting with $\sigma$ and 
$\sigma'$, and such that $u^{-1}(\fa) \to \fa'$, then $\Delta_S$ and $\Delta_{S'}$ 
commute with base change by $u$. 

\romain If $S = \Spec k$, where $k$ is a perfect field of characteristic $p$, $\Delta_S$ 
is the equivalence defined by Gieseker in \cite[Th.~1.3]{Gi}.

\end{thm}

\begin{proof}
We define the equivalence $\Delta_S$ by composing the quasi-inverse to the forgetful
functor $\Omega_S$ defined in \ref{OeqD} with the equivalence \eqref{divide}. Each of
these is clearly exact and $\sO_S$-linear. As for the compatibility with tensor
products, it follows from the fact that the $\sD$-module structure on the tensor
product over $\sO_X$ of two $\dmx{m}{X}$-modules is defined by the tensor product of
the corresponding $m$-PD-stratifications.

To check the commutation with base change, one uses again the description of the
$\dmx{m}{X}$-module structures in terms of $m$-PD-stratifications. Then this property
follows from the fact that the construction of the homomorphism $\varphi_{m,r}$ defined in
Proposition \ref{OtoD} commutes with base change.

The last assertion follows from \eqref{explicit} and from the definition given in 
\cite{Gi}, p.~4, l.~-12, thanks to the fact that the homomorphisms $\dmx{m}{X} \to 
\dix{X}$ induce isomorphisms between the submodules of operators of order $< p^{m+1}$.
\end{proof}

\subsection{}\label{Formal}
Theorem \ref{Main} implies a similar result on adic formal schemes. We now consider the
following setup. Let $\sS$ be a locally noetherian adic formal scheme, and let $\fa
\subset \sO_{\sS}$ be an ideal of definition of $\sS$. For each $r \geq 0$, we denote by
$S_r$ the closed subscheme of $\sS$ defined by $\fa^{r+1}$. We assume that $p \in \fa$,
and that we are given an endomorphism $\sigma : \sS \to \sS$ lifting the absolute
Frobenius endomorphism of $S_0$.

Let $\sX \to \sS$ be an $\fa$-adic formal scheme over $\sS$, with reduction $X_r \to
S_r$ mod $\fa^{r+1}$. We assume that $\sX$ is smooth over $\sS$, i.e., that $X_r$ is
smooth over $S_r$ for all $r$. For all $i \geq 0$, we denote again by $\sX^{(i)}$ the
pull-back of $\sX$ by $\sigma^i$ in the category of $\fa$-adic formal schemes, and we
assume that we are given a morphism $F : \sX \to \sX^{(1)}$ lifting the relative
Frobenius morphism $F_{X_0/S_0}$. We extend the notation as in \ref{Setup} to define $F
: \sX^{(i)} \to \sX^{(i+1)}$ for all $i$. Note that, in a neighbourhood of some point
$x \in \sX$, the relative dimension $d$ of $X_r$ over $S_r$ does not depend on $r$.
Then each $X_r$ is finite locally free of rank $p^d$ over $X_r^{(1)}$ in a
neighbourhood of $x$. By taking inverse limits, we get that $\sO_{\sX}$ is a finite
locally free over $\sO_{\sX^{(1)}}$, and similarly for each $\sO_{\sX^{(i)}}$ over
$\sO_{\sX^{(i+1)}}$.

We will say that an $\sO_{\sX}$-module $\sE$ is \textit{separated and complete} if $\sE
\riso \varprojlim_r \sE/\fa^r\sE$. Note that this is a local condition which depends
only on the underlying $\sO_{\sS}$-module, and that a finite direct sum of
$\sO_{\sX}$-modules is separated and complete if and only if each factor is separated
and complete. It follows that an $\sO_{\sX^{(i)}}$-module $\sF$ is separated and
complete if and only if $F^*\sF$ is separated and complete. Coherent
$\sO_{\sX}$-modules are separated and complete.

As in \ref{Setup}, we define $F$-divided $\sO_{\sX}$-modules as being families $(\sE_i,
\alpha_i)_{i \geq 0}$ of $\sO_{\sX^{(i)}}$-modules endowed with
$\sO_{\sX^{(i)}}$-linear isomorphisms $\alpha_i : F^*\sE_{i+1} \riso \sE_i$. We will
say that an $F$-divided module is separated and complete if each $\sE_i$ is separated
and complete.

Let $\dix{\sX} = \cup_n (\varprojlim_r\dix{X_r,n})$, where $\dix{X_r,n}$ is the
subsheaf of differential operators of order $\leq n$ on $X_r$. We will say that a
$\dix{\sX}$-module is separated and complete if the underlying $\sO_{\sX}$-module is
separated and complete.

\begin{thm}\label{Mainform}
Under the hypotheses of \ref{Formal}, there exists an equivalence of categories 
\ml{deltaform}{ \Delta_{\sS} : \{ \textnormal{Separated and complete\
}F\textnormal{-divided\ }\sO_{\sX}\textnormal{-modules} \} \\
 \xra{\ \approx\ } 
 \{ \textnormal{Separated and complete\ }\dix{\sX}\textnormal{-modules} \} }
which satisfies the following properties:

\romain If $(\sE_i,\alpha_i)$ is a separated and complete $F$-divided $\sO_{\sX}$-module,
the $\sO_{\sX}$-module underlying $\Delta_{\sS}(\sE_i,\alpha_i)$ is $\sE_0$.

\romain $\Delta_{\sS}$ is $\sO_{\sS}$-linear, preserves exactness for exact sequences of
separated and complete modules, and is compatible with completed tensor products.

\romain For each $r$, $\Delta_{\sS}$ induces the equivalence $\Delta_{S_r}$ between the 
subcategories of objects annihilited by $\fa^r$.

\romain If $\sS = \Spf\sV$, where $\sV$ is a complete discrete valuation ring of mixed
characteristics, $\fa$ is the maximal ideal of $\sV$, and $\sX$ is affine over $\sS$
with local coordinates, then $\Delta_{\sS}$ is the equivalence described in
\cite[3.2.2]{DS} (introduced by Matzat in the context of local differential modules
over $1$-dimensional local differential rings \cite{Ma}).

\end{thm}

\begin{proof}
Let $F_r : X_r \to X_r^{(1)}$ be the reduction of $F$ modulo $\fa^{r+1}$. Using the
$F_r$'s, Theorem \ref{Main} provides for each $r$ an equivalence between the category
of $F$-divisible $\sO_{X_r}$- modules and the category of $\dix{X_r}$-modules.
Moreover, these equivalences are compatible when $r$ varies, thanks to \ref{Main} (ii).
So, starting from a separated and complete $F$-divided $\sO_{\sX}$-module
$(\sE_i,\alpha_i)$, we obtain an inverse systeme of $F$-divided $\sO_{X_r}$-modules
$(\sE_i/\fa^{r+1}\sE_i,\alpha_i \mathrm{\ mod\ } \fa^{r+1})$, from which we deduce
compatible structures of $\dix{X_r}$-modules on the quotients $\sE_0/\fa^{r+1}\sE_0$.
Viewing these as comptatible structures of $\dix{\sX}$-modules, they define a structure
of $\dix{\sX}$-module on $\sE_0 \riso \varprojlim_r \sE_0/\fa^{r+1}\sE_0$. This defines
the functor $\Delta_{\sS}$, and property (i) is satisfied.

Conversely, given a separated and complete $\dix{\sX}$-module $\sE$, the inverse system
$(\sE/\fa^{r+1}\sE)$ defines an inverse system of $F$-divided $\sO_{X_r}$-modules
$(\sE_{i,r},\alpha_{i,r})$ such that $\sE_{0,r}=\sE/\fa^{r+1}\sE$. Since a
$\dix{\sX^{(i)}}$-module $\sF$ is separated and complete if and only if $F^*\sF$ is
separated and complete, and $F^* \varprojlim_r\sF_r \riso \varprojlim_r F_r^*\sF_r $
for any inverse system of $\sO_{X^{(i)}_r}$-modules $(\sF_r)$, we obtain by taking
inverse limits a separated and complete $F$-divided $\sO_{\sX}$-module
$(\sE_i,\alpha_i)$ such that $\sE_0 = \sE$. .

The $\sO_{\sS}$-linearity and the compatibility with completed tensor products are also
clear from the definition of $\Delta_{\sS}$, as is assertion (iii). Compatibility with
exact sequences follows from (i), since, on the one hand a sequence of
$\dix{\sX}$-modules is exact if and only if it is exact as a sequence of
$\sO_{\sX}$-modules, on the other hand a sequence of $F$-divided modules is exact if
and only if the sequence of terms of index $0$ is exact (as $F$ is finite locally
free).

To prove the last assertion, we apply the description given in \eqref{explicit} for the
action of differential operators on the reduction mod $\fa^{r+1}$, for all $r$. Indeed,
let us fix an integer $r$ and a multi-index $\uk$, and let $m$ be such that $|\uk| \leq
p^m$. Then the homomorphism $\dmx{m}{X_r} \to \dix{X_r}$ maps $\udel^{\la\uk\ra_{(m)}}$
to $\udel^{[\uk]}$. Let $(\sE_i,\alpha_i)$ be an $F$-divided $\sO_{\sX}$-module, and
let $\psi_{m,r+1} : F^{m+r+1\,*}\sE_{m+r+1} \riso \sE_0$ be the isomorphism defined by
the $\alpha_i$'s. If $x$ is a section of $\sE_0$, we can write locally $x$ as a finite
combination $x = \psi_{m,r+1}(\sum_j a_{r,j}\otimes x'_{r,j})$, where $a_{r,j} \in
\sO_{\sX}$ and $x'_{r,j} \in \sE_{m+r+1}$. Let $x_{r,j} =
\psi_{m,r+1}(F^{m+r+1\,*}(x'_{r,j})) \in \sE_0$. As the image of $x_{r,j}$ in
$\sE_0/\fa^{r+1}\sE_0$ is horizontal for the structure of $\dmx{m}{X_r}$-module, thanks
to Proposition \ref{Fhoriz}, we obtain that, for any $\uk' \neq \underline{0}$,
$\udel^{[\uk']}x_{r,j} \in \fa^{r+1}\sE_0$. Thus the Leibnitz formula implies that
$\udel^{[\uk]}x \equiv \sum_j \udel^{[\uk]}(a_{r,j})x_{r,j}$ mod $\fa^{r+1}\sE_0$. If
we let $r$ tend to infinity, we obtain
\eq{adiclimit}{ \udel^{[\uk]}x = \lim_{r\to\infty} \sum_j \udel^{[\uk]}(a_{r,j})x_{r,j}. }
This shows that the $\dix{\sX}$-module structure defined by $\Delta_{\sS}$ on $\sE_0$
coincides with the one defined in \cite[3.2.2]{DS} (see also \cite[Cor.~1.5]{Ma}).
\end{proof}


\begin{thebibliography}{99}

\bibitem{Be1}P.~Berthelot, \textit{$\mathcal{D}$-modules arithm\'etiques I.
Op\'erateurs diff\'erentiels de niveau fini}, Ann. scient. \'Ec. Norm.
Sup. {\bf 29} (1996), 185--272.

\bibitem{Be2}P.~Berthelot, \textit{$\mathcal{D}$-modules arithm\'etiques II.
Descente par Frobenius}, M\'em. Soc. Math. France {\bf 81} (2000).

\bibitem{CEB}P.~Berthelot, \textit{Introduction \`a la th\'eorie arithm\'etique des
$\mathcal{D}$-modules}, Ast\'erisque \textbf{279} (2002), 1--80.

\bibitem{BO}P.~Berthelot, A.~Ogus, \textit{Notes on crystalline cohomology}, 
Mathematical Notes \textbf{21}, Princeton Univ. Press (1978).

\bibitem{DS}J.~P.~dos~Santos, \textit{Lifting $\mathcal{D}$-modules from positive to 
zero characteristic}, preprint, to appear. 

\bibitem{Gi}D.~Gieseker, \textit{Flat vector bundles and the fundamental group in 
non-zero characteristics}, Ann. Scuola Norm. Sup. Pisa \textbf{2} (1975), 1--31.

\bibitem{EGAIV3}A.~Grothendieck, \textit{\'El\'ements de G\'eom\'etrie Alg\'ebrique} 
(r\'edig\'es avec la collaboration de J.~Dieudonn\'e) : IV. \textit{\'Etude locale des 
sch\'emas et des morphismes de sch\'emas, Troisi\`eme partie}, 
Publ. Math. I.H.\'E.S. \textbf{28} (1966), 5--255.

\bibitem{EGAIV4}A.~Grothendieck, \textit{\'El\'ements de G\'eom\'etrie Alg\'ebrique} 
(r\'edig\'es avec la collaboration de J.~Dieudonn\'e) : IV. \textit{\'Etude locale des 
sch\'emas et des morphismes de sch\'emas, Quatri\`eme partie}, 
Publ. Math. I.H.\'E.S. \textbf{32} (1967), 5--361.

\bibitem{Gr}A.~Grothendieck, \textit{Crystals and the De Rham cohomology of schemes}, 
in \textit{Dix expos\'es sur la cohomologie des sch\'emas}, North-Holland (1968), 
306--358.

\bibitem{IR}L.~Illusie, M.~Raynaud, \textit{Les suites spectrales associ\'ees au 
complexe de de Rham-Witt}, Publ. Math. I.H.\'E.S. \textbf{57} 
(1983), 73--212.

\bibitem{Ka}N.~Katz, \textit{Nilpotent connections and the monodromy 
theorem. Applications of a result of Turritin}, Publ. Math. I.H.\'E.S. \textbf{39} 
(1970), 175--232.

\bibitem{Ma}B.~H.~Matzat, \textit{Integral $p$-adic differential modules}, S\'eminaires et
Congr\`es \textbf{13} (2006), Soc. Math. France, 263--292.

\end{thebibliography}
\end{document}